\documentclass[12pt]{article}
\usepackage{amsmath,amsthm,amssymb,mathrsfs,graphicx,graphics,pdfsync}
\usepackage{eucal,esint}
\numberwithin{equation}{section}
\usepackage{color}
\usepackage[a4paper,margin=3.00truecm]{geometry}
\newcounter{marnote}

\def \dis {\displaystyle}

\def \into {\int_\Omega}
\def \confai {-\kern -.5em\rightharpoonup}

\def \ep {\varepsilon}

\def \Om {\Omega}
\def \NN {\mathbb N}

\def \RR {\mathbb R}

\def\E{\mathcal{E}}
\def\F{\mathcal{F}}

\def \M {\mathscr{M}}

\def \R {\mathscr{R}}

\def \beq {\begin{equation}}
\def \eeq {\end{equation}}
\def \ba {\begin{array}}
\def \ea {\end{array}}

\def \ecart {\noalign{\medskip}}

\newtheorem{Thm}{Theorem}[section]
\newtheorem{Cor}[Thm]{Corollary}
\newtheorem{Pro}[Thm]{Proposition}
\newtheorem{Lem}[Thm]{Lemma}

\newtheorem{Adef}[Thm]{Definition}
\newenvironment{Def}{\begin{Adef}}{\end{Adef}}
\newtheorem{Arem}[Thm]{Remark}
\newenvironment{Rem}{\begin{Arem}\rm}{\end{Arem}}
\newtheorem{Aexa}[Thm]{Example}
\newenvironment{Exa}{\begin{Aexa}\rm}{\end{Aexa}}
\newtheorem{Anot}[Thm]{Notation}

\def \reff #1.{figure~\ref{#1}}
\def \refs #1.{Section~\ref{#1}}
\def \refss #1.{Subsection~\ref{#1}}
\def \refD #1.{Definition~\ref{#1}}
\def \refT #1.{Theorem~\ref{#1}}
\def \refL #1.{Lemma~\ref{#1}}
\def \refC #1.{Corollary~\ref{#1}}
\def \refP #1.{Proposition~\ref{#1}}
\def \refPt #1.{Properties~\ref{#1}}
\def \refR #1.{Remark~\ref{#1}}
\def \refE #1.{Example~\ref{#1}}
\def \refN #1.{Notation~\ref{#1}}
\title{{\bf Optimal sources for elliptic PDEs}}

\begin{document}
\maketitle
\centerline{\large G. Buttazzo$^\dag$,\hskip 0.3cm J. Casado-D\'{\i}az$^{\dag\dag}$,\hskip 0.3cm F. Maestre$^{\dag\dag}$}
\bigskip 
\centerline{{$^\dag$} Dipartimento di Matematica, Universit\`a di Pisa,}
\centerline{Largo B. Pontecorvo, 5}
\centerline{56127 Pisa, ITALY}
\bigskip
\centerline{$^{\dag\dag}$ Dpto. de Ecuaciones Diferenciales y An\'alisis Num\'erico,}
\centerline{Facultad de Matem\'aticas, C. Tarfia s/n}
\centerline{41012 Sevilla, SPAIN}
\bigskip
\centerline{e-mail: giuseppe.buttazzo@unipi.it, jcasadod@us.es,
fmaestre@us.es}

\begin{abstract}
We investigate optimal control problems governed by the elliptic partial differential equation $-\Delta u=f$ subject to Dirichlet boundary conditions on a given domain $\Om$. The control variable in this setting is the right-hand side $f$, and the objective is to minimize a cost functional that depends simultaneously on the control $f$ and on the associated state function $u$.

We establish the existence of optimal controls and analyze their qualitative properties by deriving necessary conditions for optimality. In particular, when pointwise constraints of the form $\alpha\le f\le\beta$ are imposed a priori on the control, we examine situations where a {\it bang-bang} phenomenon arises, that is where the optimal control $f$ assumes only the extremal values $\alpha$ and $\beta$. More precisely, the control takes the form $f=\alpha1_E+\beta1_{\Om\setminus E}$, thereby placing the problem within the framework of shape optimization. Under suitable assumptions, we further establish certain regularity properties for the optimal sets $E$.

Finally, in the last part of the paper, we present numerical simulations that illustrate our theoretical findings through a selection of representative examples.
\end{abstract}


\medskip


%%%%%%%%%%%%%%%%%%%%%%%%%%%%%%%%%%%%%%%%%%%%%%%%%%
\section{Introduction}\label{intro}

In this paper, we study an optimal control problem for a partial differential equation governed by the Laplace operator in a given bounded domain $\Om$ of $\RR^d$, with homogeneous Dirichlet boundary conditions on $\partial\Om$. The control variable is the right-hand side $f$, which is required to lie within a suitably chosen admissible class $\F$. The associated state equation reads
\beq\label{rhs}
\begin{cases}
-\Delta u=f&\hbox{in }\Om\\
u\in H^1_0(\Om).
\end{cases}\eeq
and we denote by $u_f$ the unique weak solution corresponding to a given control $f$.

The cost functional to be minimized is of the form
\beq\label{J}
J(f)=\int_\Om j(x,u_f,f)\,dx,
\eeq
where $j$ is a prescribed integrand satisfying appropriate conditions. The optimal control problem can thus be formulated as
$$\min\big\{J(f)\ :\ f\in\F\big\}.$$
We focus on the case where the admissible class $\F$ is defined via an integral constraint of the type
$$\F=\bigg\{\int_\Om\psi(f)\,dx\le m\bigg\},$$
for some given $m>0$ and a convex lower semicontinuous function $\psi:\RR\to [0,\infty]$ satisfying the following hypotheses:
$$\begin{cases}
{\rm int}(D(\psi))\not=\emptyset\ \hbox{ with }D(\psi)=\big\{s\in\RR:\ \psi(s)<\infty\}\\
\dis\lim_{|s|\to+\infty}\psi(s)=+\infty.
\end{cases}$$
Under these assumptions, the optimization problem we deal with takes the form \beq\label{optpb}
\min\bigg\{\int_\Om j(x,u_f,f)\,dx\ :\ \int_\Om\psi(f)\,dx\le m\bigg\},
\eeq

A particularly interesting case arises when the control $f$ is constrained to lie between two prescribed constants $\alpha$ and $\beta$. This constraint can be expressed by taking
$$\psi(s)=+\infty\qquad\text{if }s\notin[\alpha,\beta].$$

Under this setting, and for suitable choices of the integrand $j$ in the cost functional, a {\it bang-bang} phenomenon may occur, meaning that the optimal control $f$ attains only the extreme values $\alpha$ and $\beta$. More precisely, the optimal control takes the form
$$f=\beta1_E+\alpha1_{\Om\setminus E}$$
for some measurable subset $E\subset\Om$. In this regime, the problem naturally transforms into a {\it shape optimization problem}, where the control variable is the set $E$ itself. We devote particular attention to this case, discussing several related aspects, including the regularity properties of the optimal sources $f$ and the structural features of the associated optimal sets $E$.

Finally, in Section~\ref{snum}, we present a series of numerical simulations that illustrate the theoretical phenomena described and provide concrete examples of the optimal configurations.

%%%%%%%%%%%%%%%%%%%%%%%%%%%%%%%%%%%%%%%%%%%%%%%%%%
\section{Notation}\label{notation}

In this section, for the convenience of the reader, we introduce and summarize the main notation that will be consistently used throughout the paper.

\begin{itemize}
\item We denote by $\Om$ a bounded domain in $\RR^d$.

\item Let $\psi:\RR\to (-\infty,\infty]$ be a convex lower semicontinuous function. We introduce the following related notions:
\begin{itemize}
\item The domain of $\psi$, denoted by $D(\psi)$, is defined by
$$D(\psi)=\big\{s\in\RR:\ \psi(s)<\infty\big\}.$$
\item The conjugate function $\psi^*:\RR\to(-\infty,\infty]$ is given by
$$\psi^*(t):=\sup_{s\in D(\psi)}\big(ts-\psi(s))$$
\item The subdifferential of $\psi$ at a point $s\in D(\psi)$, denoted by $\partial\psi(s)$, is defined as
$$\partial\psi(s)=\big\{\xi\in\RR,\quad\psi(r)\ge\psi(s)+\xi(r-s),\ \forall\,r\in\RR\big\}=\big[d_-\psi(s),d_+\psi(s)\big],$$
where
$$d_-\psi(s)=\lim_{r\nearrow s}{\psi(r)-\psi(s)\over r-s},\qquad d_-\psi(s)=\lim_{r\searrow s}{\psi(r)-\psi(s)\over r-s}$$
denote, respectively, the left and right derivatives of $\psi$ at $s$.
\item The recession limits of $\psi$, denoted by $c^-(\psi)$ and $c^+(\psi)$, are defined by
$$c^-(\psi)=\lim_{s\to-\infty}\frac{\psi(s)}{s}\qquad c^+(\psi)=\lim_{s\to+\infty}\frac{\psi(s)}{s}\;.$$
\end{itemize}

\item For a bounded open set $\Om\subset\RR^d$, we denote by ${\cal M}(\Om)$ the space of bounded Borel measures on $\Om$.

\item Given $f\in{\cal M}(\Om)$, we denote by $f^a$ and $f^s$ the absolutely continuous and singular parts of $f$ in its Radon-Nikodym decomposition:
$$f=f^adx+f^s.$$
The positive and negative parts of a measure $f$ are denoted by $f_-$ and $f_+$ respectively. The support of a measure $f$ is denoted by ${\rm supp}(f)$.

\item For $s,t\in\RR$, we denote by $t\wedge s$ and $t\vee t$ the minimum and maximum of $s$ and $t$, respectively.

\item For any $m>0$, we define the truncation function $T_m:\RR\to [-m,m]$ at height $m$, by
$$T_m(s)=(m\wedge s)\vee(-m),\quad\forall\,s\in\RR.$$
\end{itemize}

%%%%%%%%%%%%%%%%%%%%%%%%%%%%%%%%%%%%%%%%%%%%%%%%%%
\section{Existence of an optimal source}\label{sexi}

In this section, we establish the existence of an optimal source term $f$ under suitable mild assumptions. We begin by considering the case where the function $\psi$ is convex and exhibits superlinear growth at infinity, that is
\beq\label{super}
\lim_{|s|\to+\infty}\frac{\psi(s)}{|s|}=+\infty.
\eeq

\begin{Thm}\label{ex1}
Suppose that the functional \eqref{J} is lower semicontinuous with respect to the weak $L^1(\Om)$ topology, and that the integrand $j(x,s,z)$ satisfies the growth condition \beq\label{pgrj}
-c|s|^p-a(x)\le j(x,s,z),\qquad\text{with }c>0,\ a\in L^1(\Om),\ p<d/(d-2).
\eeq
If, in addition, the function $\psi$ satisfies the superlinear growth condition \eqref{super}, then the optimization problem \eqref{optpb} admits at least one solution $f_{opt}\in L^1(\Om)$.
\end{Thm}

\begin{proof}
Assuming that $\psi$ grows superlinearly, any minimizing sequence $(f_n)$ for the optimization problem \eqref{optpb} is relatively compact in the weak topology of $L^1(\Om)$. Thus, up to a subsequence, we may suppose that $f_n\to f$ weakly in $L^1(\Om)$ for some $f\in L^1(\Om)$.

Moreover, due to the compact embedding of $L^1(\Om)$ into $W^{-1,q}(\Om)$ for every $q<d/(d-1)$, the corresponding solutions $u_n$ to the PDEs \eqref{rhs} converge strongly in $W^{1,q}_0(\Om)$, and hence strongly in $L^p(\Om)$ for all $p<d/(d-2)$, to the solution $u$ associated with the limit $f$.

Finally, by the lower semicontinuity of the mappings
$$f\mapsto J(f),\qquad\text{and}\qquad f\mapsto\int_\Om\psi(f)\,dx,$$
with respect to the weak $L^1(\Om)$ topology, it follows that $f$ indeed minimizes the original functional. Consequently, $f$ is an optimal solution.
\end{proof}

\begin{Rem} A sufficient condition ensuring the weak $L^1(\Om)$ lower semicontinuity of the functional $J$ defined in \eqref{J} is that the integrand $j(x,\cdot,\cdot)$ is lower semicontinuous in its arguments for almost every $x$, and that $j(x,s,\cdot)$ is convex for almost every $x$ and every $s$. For further details, we refer the reader to \cite{B}.
\end{Rem}

\begin{Rem}\label{rqgr}
If we strengthen the growth assumption on $\psi$ by requiring that there exists $q>1$ such that
\beq\label{concqpsi}
c|s|^q-a\le\psi(s)\qquad\text{for some }c>0,\ a\in\RR,
\eeq
then the growth condition \eqref{pgrj} on the integrand $j$ can be accordingly relaxed and allows for broader classes of nonlinearities and source terms, adapting to the growth properties of $\psi$. Specifically, we may assume:
$$\begin{cases}
\dis-c|s|^p-a(x)\le j(x,s,z),\quad\text{with }c>0,\ a\in L^1(\Om),\ p<{dq\over d-2q}&\text{if }q<d/2\\
\dis-ce^{|s|^p}-a(x)\le j(x,s,z),\quad\text{with }c>0,\ a\in L^1(\Om),\ \ p<{d\over d-1}&\text{if }q=d/2\\
-a_n(x)\le j(x,s,z) \ \hbox{ for }|s|<n,\quad\text{with }a_n\in L^1(\Om),\ \forall\,n\in\NN&\text{if }q>d/2.
\end{cases}$$
\end{Rem}

We now turn our attention to the case when the function $\psi$ exhibits a linear growth, that is,
\beq\label{linear}
c|s|-a\le\psi(s)\qquad\text{for some constants }c>0,\ a\in\RR.
\eeq
In this setting, the optimal source term may no longer belong to $L^1(\Om)$, but may instead be represented by a finite Radon measure. Accordingly, the integral $\int_\Om\psi(f)$ must be interpreted in the sense of measures, namely:
\beq\label{psim}
\int_\Om\psi(f)=\int_\Om\psi\big(f^a(x)\big)\,dx+c^+(\psi)\int df_+^s-c^-(\psi)\int df_-^s.
\eeq
It is a classical result that functionals of the form \eqref{psim} are lower semicontinuous with respect to the weak* convergence of measures.

\begin{Thm}\label{ex2}
Suppose that the functional \eqref{J} is weakly* lower semicontinuous in the space $\M(\Om)$ of finite Radon measures, and that the integrand $j$ satisfies the growth condition
$$-c|s|^p-a(x)\le j(x,s,z),\qquad\text{for some }c>0,\ a\in L^1(\Om),\ p<d/(d-2).$$
If, in addition, the function $\psi$ satisfies the linear growth condition \eqref{linear}, then the optimization problem \eqref{optpb} admits at least one optimal solution $f_{opt}$, which is a measure with finite total variation.
\end{Thm}

\begin{proof}
The proof proceeds along similar lines as that of Theorem \ref{ex1}. Let $(f_n)$ be a minimizing sequence for the optimization problem \eqref{optpb}. Since $(f_n)$ is bounded in the space of finite Radon measures, by the Banach-Alaoglu theorem, we can extract a subsequence (still denoted by $(f_n)$) which converges to some measure $f$ in the weak* topology of $\M(\Om)$.

The corresponding sequence of solutions $(u_n)$ to the PDEs \eqref{rhs} then converges strongly in $W^{1,q}_0(\Om)$ for every $q<d/(d-1)$, and therefore also strongly in $L^p(\Om)$ for every $p<d/(d-2)$, to the solution $u$ associated with the limit measure $f$.

Finally, the weak* lower semicontinuity of both terms involved in the optimization problem \eqref{optpb},
$$f\mapsto J(f),\qquad\text{and}\qquad f\mapsto\int_\Om\psi(f),$$
ensures that $f$ is indeed an optimal solution to \eqref{optpb}.
\end{proof}

\begin{Rem}
A sufficient condition for the lower semicontinuity of the functional $J$ in \eqref{J} with respect to the weak* convergence of measures is the following (see for example \cite{BB}). Suppose the integrand $j(x,s,z)$ admits the decomposition in the form
$$j(x,s,z)=A(x,s)+B(x,z),$$
where the functions $A$ and $B$ satisfy the following properties:
\begin{itemize}
\item[-]for almost every $x\in\Om$ the function $A(x,\cdot)$ is lower semicontinuous;
\item[-]there exist constants $c>0$, $p<d/(d-2)$ and a function $a\in L^1(\Om)$ such that
$$A(x,s)\ge-c|s|^p+a(x);$$
\item[-]for almost every $x\in\Om$ the function $B(x,\cdot)$ is convex and lower semicontinuous;
\item[-]the associated recession function
$$B^\infty(x,z)=\lim_{t\to+\infty}\frac{B(x,tz)}{t}$$
is lower semicontinuous with respect to both variables $(x,z)$;
\item[-]there exist functions $a_0\in C_0(\Om)$ and $a_1\in L^1(\Om)$ such that
$$B(x,z)\ge a_0(x)z+a_1(x).$$
\end{itemize}
The assumptions on the function $A$ allow to obtain the lower semicontinuity thanks to the Fatou's lemma, while the assumptions on the function $B$ allow to obtain the lower semicontinuity thanks to the results on functionals defined on measures. For all the details we refer to \cite{BB}, where more general cases, including the ones where the functional $J$ is not convex, are considered.
\end{Rem}

%%%%%%%%%%%%%%%%%%%%%%%%%%%%%%%%%%%%%%%%%%%%%%%%%%
\section{Necessary conditions of optimality}\label{snec}

In this section, we derive some necessary conditions of optimality that any solution $f_{opt}$ must satisfy. These conditions are presented in Theorem \ref{Thcop} below. To this end, it is convenient to introduce the resolvent operator $\R$, which associates to every function $f$ the unique solution $u$ of the partial differential equation \eqref{rhs}. It is well known that $\R$ is a self-adjoint operator.

\begin{Thm}\label{Thcop}
Suppose that the function $j$ appearing in the formulation of the optimal control problem \eqref{optpb} satisfies the growth condition
$$|j(x,s,z)|\le a(x)+c|s|^p,\qquad\hbox{with }c>0,\ a\in L^1(\Om),\ p<d/(d-2).$$
In addition, we assume that one of the following conditions holds.
\begin{itemize}
\item (Case of superlinear growth): If $\psi$ satisfies the superlinear growth condition \eqref{super}, then for almost every $x\in\Om$ and every $(s,z)\in\RR^2$, the partial derivatives $\partial_sj(x,s,z)$ and $\partial_zj(x,s,z)$ exist and fulfill
\beq\label{acdjcs}
\begin{cases}
|\partial_s j(x,s,z)|\le b(x)+\gamma\big(|s|^\sigma+|z|^\tau\big)\\
|\partial_z j(x,s,z)|\le\gamma,
\end{cases}\eeq
where $\gamma>0$, $b\in L^q(\Om)$ with $q>d/2$, $\sigma<2/(d-2)$, and $\tau< 2/d$.
\item (Case of linear growth): If $\psi$ exhibits a linear growth, meaning $c^+(\psi)-c^-(\psi)>0$, then $j=j(x,s,z)$ depends only on $(x,s)$ and not on $z$. In this case, for almost every $x\in\Om$ and every $s\in\RR$, the partial derivative $\partial_sj(x,s)$ exists and satisfies
\beq\label{acdjcl}
|\partial_s j(x,s)|\le b(x)+\gamma|s|^\sigma,
\eeq
where again $\gamma>0$, $b\in L^q(\Om)$ with $q>d/2$, $\sigma<2/(d-2)$.
\end{itemize}
Then, if $f_{opt}$ is an optimal solution to the problem \eqref{optpb}, there exists a non-negative scalar $\lambda\ge0$ such that
\beq\label{condpl}
\lambda\left(\into\psi(f_{opt})dx-m\right)=0,
\eeq
and, setting
\beq\label{defw}
w:=\R\big(\partial_sj(x,\R(f_{opt}),f_{opt})\big)+\partial_zj(x,\R(f_{opt}),f_{opt}),
\eeq
the following alternative holds:
\begin{itemize}
\item If $\lambda=0$, then
\beq\label{condoptl0} 
\begin{cases}
w\ge0\hbox{ a.e. in }\Om\hbox{ if }\sup\big(D(\psi)\big)=+\infty\\
w\le0\hbox{ a.e. in }\Om\hbox{ if }\inf\big(D(\psi)\big)=-\infty\\
f_{opt}^a=\min\big(D(\psi)\big)\hbox{ a.e. in }\big\{w> 0\big\}\\
f_{opt}^a=\max\big(D(\psi)\big)\hbox{ a.e. in }\big\{w< 0\big\}\\
{\rm supp}(f^s_{opt})\subset \{w=0\}.
\end{cases}\eeq
\item If $\lambda>0$, then
\beq\label{condopl>0}
\begin{cases}
\dis\psi\big(f_{opt}^a\big)+\psi^*\big(-{w\over\lambda}\big)=-{wf^a_{opt}\over\lambda}\hbox{ a.e. in }\Om\\
-\lambda c^+(\psi)\le w\le -\lambda c^-(\psi)\hbox{ a.e. in }\Om\\
{\rm supp}(f_{opt,+}^s)\subset\big\{w+\lambda c^+(\psi))=0\big\}\\
{\rm supp}(f_{opt,-}^s)\subset\big\{w+\lambda c^-(\psi))=0\big\}.
\end{cases}\eeq
\end{itemize}
Moreover, if the function $j(x,.,.)$ is convex for almost every $x\in\Om$, then the conditions stated above are not only necessary for optimality but also sufficient. \end{Thm}

\begin{proof}
Since the function $\psi$ is convex, for any $f\in{\cal M}(\Om)$ satisfying the constraint $\into\psi(f)\,dx\le m$, the mapping
$$\ep\in[0,1]\mapsto\into j\big(x,\R(f_{opt}+\ep(f-f_{opt})),f_{opt}+\ep(f-f_{opt})\big)\,dx$$
attains its minimum at $\ep=0$. Thanks to the regularity assumptions \eqref{acdjcs} or \eqref{acdjcl}, combined with the fact that $\R(f_{opt})\in L^r(\Om)$ for every $r\in [1,d/(d-2)]$, we can differentiate under the integral sign with respect to $\ep$ at $\ep=0$, leading to
\[\begin{split}
0&\le\into\Big(\partial_s j\big(x,\R(f_{opt}),f_{opt}\big)\R(f-f_{opt})+\partial_z j\big(x,\R(f_{opt}),f_{opt}\big)(f-f_{opt})\Big)\,dx\\
&=\into\Big(\R\big(\partial_s j\big(x,\R(f_{opt}),f_{opt}\big)\big)+\partial_z j\big(x,\R(f_{opt}),f_{opt}\big)\Big)(f-f_{opt})\Big)\,dx\\
&=\into w(f-f_{opt})\,dx,
\end{split}\]
where, recalling \eqref{defw}, we have set
$$w=\R\big(\partial_sj(x,\R(f_{opt}),f_{opt})\big)+\partial_zj(x,\R(f_{opt}),f_{opt}).$$
Thus, we deduce that $f_{opt}$ solves the following convex minimization problem: \beq\label{derifu}
\min\bigg\{\into wf\,dx\ :\ \into\psi(f)\le m\bigg\}.
\eeq
Applying the Kuhn-Tucker theorem, we infer the existence of a Lagrange multiplier $\lambda\ge0$ satisfying the complementary condition \eqref{condpl}, such that $f_{opt}$ is a solution to
\beq\label{conskt}
\begin{cases}
\dis\min\bigg\{\into wf\,dx+\lambda\into\psi(f)\,dx\ : \ f\in {\cal M}(\Om)\bigg\}&\hbox{if }\lambda>0\\
\dis\min\bigg\{\into wf\,dx\ :\ f\in {\cal M}(\Om),\ f^a\in D(\psi)\hbox{ a.e. in }\Om\bigg\}&\hbox{if }\lambda=0.
\end{cases}
\eeq
In particular, this shows that, almost everywhere in $\Om$, the absolutely continuous part $f^a_{opt}(x)$ solves the following pointwise minimization problem:$$\begin{cases}
\dis\min_{s\in\RR}\big\{w(x)s+\lambda\psi(s)\big\}&\hbox{if }\lambda>0\\
\dis\min_{s\in D(\psi)}w(x)s&\hbox{if }\lambda=0,
\end{cases}$$
thereby establishing the first four conditions in \eqref{condoptl0} and the first condition in \eqref{condopl>0}.

\medskip

Let us now assume that $c^+(\psi)>0$ (hence $w\in C^0(\overline\Om)$). Suppose by contradiction that there exists $x\in \overline\Om$ such that $w(x)+\lambda c^+(\psi)<0$. Then, considering the test measure $f=n\delta_{x}$ with $n>0$, and letting $n\to\infty$, we observe that the value of the minimization problem \eqref{conskt} would tend to $-\infty$, contradicting the existence of an optimal solution $f_{opt}$. Consequently, we must have $w(x)+\lambda c^+(\psi)\ge0$ for all $x\in\overline\Om$.

Furthermore, noting that for any nonnegative singular measure $f^s$ it holds
$$0\le\into(w+\lambda c^+(\psi))\,df^s_+,$$
we deduce that the support of the positive part of the singular component satisfies
$${\rm supp}(f^s_{opt,+})\subset\big\{w+\lambda c^+(\psi)=0\big\}.$$
A similar argument, considering the case $c^-(\psi)<0$, yields that
$$\begin{cases}
w+\lambda c^-(\psi)\le0\quad\hbox{in }\overline\Om,\\
{\rm supp}(f^s_{opt,-})\subset\big\{w+\lambda c^-(\psi)=0\big\}.
\end{cases}$$

\medskip

Finally, when $j(x,\cdot,\cdot)$ is convex for almost every $x\in\Om$, the original optimization problem \eqref{optpb} is itself convex. In this case, $f_{opt}$ solves \eqref{optpb} if and only if it solves the equivalent convex minimization problem \eqref{derifu}, and thus if and only if the necessary optimality conditions stated in Theorem \ref{Thcop} are satisfied.
\end{proof}

\begin{Rem}\label{coosbd}
The first condition in \eqref{condopl>0} can equivalently be reformulated in either of the following forms:
\beq\label{foecelp}
-w\in\lambda\partial\psi(f_{opt}^a)\hbox{ a.e. in }\Om\qquad\hbox{or}\qquad f_{opt}^a\in\partial\psi^*\big(-{w\over\lambda}\big)\ \hbox{ a.e. in }\Om.\eeq
The second formulation provides a characterization of the optimal control $f^a_{opt}$ directly in terms of the adjoint variable $w$.

In the present work, our primary interest is focused on the case where the optimal control $f_{opt}^a$ exhibits a bang-bang structure. According to the second condition in \eqref{foecelp}, such a behavior arises if there exists a point $s\in{\rm int}(D(\psi^*))$ where the convex conjugate $\psi^*$ fails to be differentiable. More precisely, under this assumption, we have
\beq\label{ndpas}
\partial\psi^*(s)=[d_-\psi^*(s),d_+\psi^*(s)],\qquad-\infty<d_-\psi^*(s)<d_+\psi^*(s)<\infty,\eeq
which leads to the following characterization:
\beq\label{fadiw}
\begin{cases}
f^a_{opt}(x)\ge d_+\psi^*(s)&\hbox{if }w(x)<-\lambda s\\
f^a_{opt}(x)\le d_-\psi^*(s)&\hbox{if }w(x)>-\lambda s.
\end{cases}\eeq
It is important to note that if the set $\{x\in\Om\ :\ w(x)=-\lambda s\}$ has a positive Lebesgue measure, then condition \eqref{fadiw} does not necessarily imply that $f^a_{opt}$ is discontinuous on this set.

Assuming furthermore that the function $j(x,s,z)$ is independent of $z$, and recalling that the function $w=\R(\partial_s j(x,\R(f_{opt})))$ belongs to $W^{2,q}_{loc}(\Om)$, it follows that $\Delta w=0$ almost everywhere in $\{w=s\}$, for every $s\in\RR$. Consequently, we obtain:
\beq\label{dsjnn}
\big|\big\{\partial_sj(x,\R(f_{opt}))=0\}\big|=0\Longrightarrow\big|\{w=s\}\big|=0,\quad\forall\,s\in \RR.
\eeq 
A particularly simple sufficient condition to ensure \eqref{dsjnn} is that the map $s\mapsto j(x,s)$ be either strictly increasing or strictly decreasing for each $x\in\Om$.

On the other hand, it is useful to recall that condition \eqref{ndpas} is equivalent to the relation
$$\psi(t)=st-\psi^*(s)\qquad\forall\,t\in[d_-\psi^*(s),d_+\psi^*(s)],$$
meaning that $\psi$ must be affine on an interval of positive length. Therefore, a necessary condition on the function $\psi$ for the appearance of bang-bang optimal controls is the existence of a bounded interval with nonempty interior on which $\psi$ is affine, that is, the function $\psi$ must fail to be strictly convex over some nontrivial subinterval.
\end{Rem}

\begin{Rem}
By an argument similar to the one of Remark \ref{rqgr}, the growth conditions imposed on the function $j$ and its derivatives in Theorem \ref{Thcop} can be relaxed when the function $\psi$ satisfies the condition \eqref{concqpsi}. Specifically, when $q>d/2$, it suffices to require that, for every $n>0$,
$$|j(x,s,z)|+|\partial_sj(x,s,z)|\le a_n(x)+c_n|z|^q\qquad\hbox{for }|s|<n,$$
where $a_n\in L^1(\Om)$ and $c_n>0$ are given, and similarly,
$$|\partial_zj(x,s,z)|\le b_n(x)+\gamma_n|z|^{q-1}\qquad\hbox{for }|s|<n,$$
with $b_n\in L^{q/(q-1)}(\Om)$ and $\gamma_n>0$.
\end{Rem}

We are now ready to illustrate the application of Theorem \ref{Thcop} through several important examples of the function $\psi$.

\begin{Exa}
Let us now consider the case where $\psi(s)=|s|$. In this setting, problem \eqref{optpb}, under the assumption that $j(x,s,z)$ is independent of $z$ and satisfies the growth conditions \eqref{acdjcl}, can be rewritten as:
\beq\label{pbme1}
\min\bigg\{\int_\Om j(x,\R(f))dx\ :\ \|f\|_{\M(\Om)}\le m\bigg\}.\eeq
In order to apply Theorem \ref{Thcop} together with the characterization provided in Remark \ref{coosbd}, we first observe the properties of the convex conjugate $\psi^*$, namely:
$$\psi^*(t)=\begin{cases}
0&\hbox{if }t\in[-1,1]\\
+\infty&\hbox{otherwise,}
\end{cases}\qquad\partial\psi^*(t)=\begin{cases}
[-\infty,0]&\hbox{if }t=-1\\
0&\hbox{if }t\in(-1,1)\\
[0,\infty]&\hbox{if }t=1.
\end{cases}$$
If $\lambda=0$ in the framework of Theorem \ref{Thcop}, then, according to condition \eqref{condoptl0} and the fact that $D(\psi)=\RR$, the optimality system simply reduces to
$$w=\R\big(\partial_sj(x,\R(f_{opt})\big)=0\qquad\text{almost everywhere in }\Om,$$ which is equivalent to the condition:
$$\partial_sj(x,\R(f_{opt}))=0\qquad\hbox{a.e. in }\Om.$$
Let us assume now that we are not in this degenerate case, so that $\lambda>0$. In this case, Theorem \ref{Thcop} combined with the optimality conditions \eqref{foecelp} yield the following set of properties:
$$\begin{cases}
-\lambda\le w\le\lambda\hbox{ a.e. in }\Om,\\
{\rm supp}(f_{opt})\subset\{|w|=\lambda\},\\
f_{opt}\ge0\hbox{ in }\{w=-\lambda\},\\
f_{opt}\le0\hbox{ in }\{w=\lambda\},\\
\|f\|_{{\cal M}(\Om)}=m.
\end{cases}$$
In particular, let us consider the situation where the function $s\mapsto j(x,s)$ is non-decreasing for almost every $x\in\Om$. In this case, we have $\partial_s j(x,\cdot)\ge0$, which, by the maximum principle applied to $w$, implies that $w\ge0$ almost everywhere in $\Om$. Therefore, $w$ satisfies $0\le w\le\lambda$ a.e. in $\Om$, and the support of the optimal control is contained in the set $\{w=\lambda\}$, with $f_{opt}\le0$.

For instance, if $\Om$ is a ball centered at the origin and $j(x,s)=s$, the solution simplifies further, and the optimal control is given explicitly by:
$$f_{opt}=-m\delta_0.$$
where $\delta_0$ denotes the Dirac mass at the origin.

An entirely similar analysis can be carried out when $j(x,\cdot)$ is non-increasing, leading to the symmetric case.
\end{Exa}

\begin{Exa}\label{ex45}
In connection with problem \eqref{pbme1}, let us now consider the variational problem \beq\label{pbme2}
\min\bigg\{\int_\Om j(x,\R(f))\,dx\ :\ f\ge0,\ \into f\,dx\le m\bigg\}.
\eeq
In this context, the function $\psi$ is given by
$$\psi(s)=\begin{cases}
s&\hbox{if }s\ge0\\
+\infty&\hbox{if }s<0,
\end{cases}$$
and its convex conjugate $\psi^*$ takes the form
$$\psi^*(t)=\begin{cases}
0&\hbox{if }t\le1\\
\infty&\hbox{if }t>1,
\end{cases}$$
with
$$\partial\psi^*(t)=\begin{cases}
0&\hbox{if }t<1\\
[0,\infty]&\hbox{if }t=1.
\end{cases}$$
Let $f_{opt}$ be an optimal solution to problem \eqref{pbme2}. Then, by applying the optimality conditions \eqref{condoptl0} and \eqref{condopl>0}, we infer the existence of a Lagrange multiplier $\lambda\ge0$ such that
$$\begin{cases}
\dis\lambda\Big(\into\psi(f_{opt})dx-m\Big)=0\\
w\ge-\lambda\hbox{ a.e. in }\Om\\
{\rm supp}(f_{opt})\subset\{w=-\lambda\}.
\end{cases}$$
This result admits a more refined characterization under additional assumptions. Suppose that for almost every $x\in\Om$, the function $s\mapsto j(x,s)$ is strictly concave. In that case, the optimal source $f_{opt}$ must be an extremal point of the admissible set
$$\bigg\{f\ge0\ :\ \int_\Om f\,dx\le m\bigg\}.$$
Consequently, the optimal solution must be a singular measure supported at a point, that is, a multiple of a Dirac delta. Assume furthermore that for almost every $x\in\Om$, the function $j(x,\cdot)$ attains its maximum at $s=0$. Since $f_{opt}\ge0$, it follows that $\R(f_{opt})\ge0$, and hence,
$$\partial_sj(x,\R(f_{opt}))\le\partial_sj(x,0)\le0\qquad\hbox{a.e. in }\Om,$$
implying that the adjoint state $w=\R(\partial_s j(x,\R(f_{opt})))\le0$ almost everywhere in $\Om$. The case where $w=0$ a.e. leads to a contradiction, as it would imply $\R(f_{opt})=0$ a.e., which would in turn correspond to the maximum, not the minimum, of the functional in \eqref{pbme2}. Thus, we conclude that the Lagrange multiplier $\lambda$ must be strictly positive, and we obtain the refined optimality condition:
\beq\label{cone2cc}
-\lambda\le w\le0\ \hbox{ a.e. in }\Om,\qquad f_{opt}=m\delta_{x_0}\quad\hbox{with }w(x_0)=-\lambda.
\eeq
As a concrete example, consider the maximization problem
\beq\label{maxlp}
\max\bigg\{\int_\Om|\R(f)|^p\,dx\ :\ f\geq 0,\quad \into f\,dx\le m\bigg\}.
\eeq
It is readily seen that if $p\ge d/(d-2)$, the functional is unbounded above and the supremum is infinite, hence no optimal solution exists. However, when $p<d/(d-2)$, the problem admits a solution $f_{opt}$, and it satisfies the structure described in \eqref{cone2cc}.
\end{Exa}

\begin{Exa}\label{exbang}
Let us consider the following optimization problem:
\beq\label{pbme3}\min\bigg\{\int_\Om j\big(x,\R(f),f\big)\,dx\ :\ \int_\Om f\,dx\le m,\ \alpha\le f\le\beta\bigg\},\eeq
subject to the bounds 
\beq\label{coabe3}
\alpha|\Om|<m\le\beta|\Om|,
\eeq
where $\alpha$ and $\beta$ are real constants. Without loss of generality, and to simplify the exposition, we assume $\alpha\ge0$; the treatment of other cases (e.g., when $\alpha<0$) follows in a similar way. The admissible set is naturally associated with the function
$$\psi(s)=\begin{cases}
s&\text{if }s\in[\alpha,\beta]\\
+\infty&\text{otherwise,}
\end{cases}$$
whose convex conjugate is given by
$$\psi^*(t)=\begin{cases}
(t-1)\alpha&\text{if }t\le1\\
(t-1)\beta&\text{if }t\ge1,
\end{cases}$$
with
$$\partial\psi^*(t)=\begin{cases}
\alpha&\text{if }t<1\\
[\alpha,\beta]&\text{if }t=1\\
\beta &\text{if }t>1.
\end{cases}$$
By Theorem \ref{Thcop}, any optimal solution $f_{opt}$ of problem \eqref{pbme3} must satisfy the pointwise condition
\beq\label{fopte3}
f_{opt}=\begin{cases}
\beta&\hbox{if }w<-\lambda\\
\alpha&\hbox{if }w>-\lambda,
\end{cases}\eeq
where $w$ is the adjoint state defined via \eqref{defw}, and $\lambda\ge0$ is a Lagrange multiplier associated with the volume constraint, satisfying the complementary condition
\beq\label{cosae3}\lambda\left(\into f\,dx-m\right)=0.\eeq
Since the adjoint variable $w$ is known to vanish on the boundary $\partial\Om$ due to the properties of $\R$, the structure of the optimal solution $f_{opt}$ is particularly simple when the following conditions occur:
$$\lambda>0,\qquad\big|\big\{w<-\lambda\}\big|>0,\qquad\big|\big\{w=-\lambda\}\big|=0.$$
Under these hypotheses, the optimal control $f_{opt}$ is of bang-bang type; that is, it takes only the extremal values $\alpha$ and $\beta$ almost everywhere in $\Om$.

Let us now examine how the qualitative nature of $f_{opt}$ depends on the structure of the integrand $j$. Assume that the function $j(x,s,z)$ is independent of $z$ and is either non-decreasing or non-increasing in the variable $s$. In the first case, where $j$ is non-decreasing in $s$, the adjoint state is non-negative:
$$w=\R(\partial_sj(x,\R(f_{opt})))\ge0.$$
Then, from \eqref{fopte3}, it follows that $f_{opt}=\alpha$ almost everywhere in $\Om$.

In contrast, if $j$ is non-increasing in $s$, then $w\le0$ a.e. in $\Om$. Suppose that the measure of the set $\{w<-\lambda\}$ is zero. Then, again from \eqref{fopte3}, we have $f_{opt}=\alpha$ a.e., and so
$$\into f_{opt}\,dx=\alpha |\Om|<m,$$
which implies, by \eqref{cosae3}, that $\lambda=0$. Consequently, the adjoint state $w$ must vanish identically, and $\partial_s j(x,\R(\alpha))=0$ as well. This is only possible if for a.e. $x\in\Om$ the function $j(x,\cdot)$ is constant in the interval $[\R(\alpha),0]$ or in the interval $[0,\R(\alpha)]$ (depending on the sign of $\alpha$). 

If this constancy condition is not satisfied, then necessarily $\lambda>0$, and the volume constraint $\into f\,dx\le m$ is saturated. In this situation, the function $f_{opt}$ takes both values $\alpha$ and $\beta$, as described by \eqref{fopte3}. In particular, this occurs whenever $j(x,\cdot)$ is strictly decreasing, in which case the condition $|\{w=-\lambda\}|=0$ is also satisfied, and $f_{opt}$ is indeed a bang-bang control.
\end{Exa}

\begin{Exa}\label{exbang2}
Another interesting example corresponds to 
$$\min\bigg\{\int_\Om\big|\R(f)-u_0\big|^2\,dx\ :\ \int_\Om f\,dx\le m,\  \alpha\le f\le\beta \bigg\},$$
with $u_0\in L^2(\Om)$ prescribed and $m$ satisfying (\ref{coabe3}). This case has been studied, with $\alpha=0$ and $\beta=1$, in \cite{LTZ}. Since this functional is strictly convex, the solution is unique and (\ref{fopte3}), (\ref{cosae3}) are necessary and sufficient conditions for $f_{opt}$, where now $w=2\R\big(\R(f_{opt})-u_0\big)$.

Since $f_{opt}\in[\alpha,\beta]$, we have $\R(f_{opt})\in[\R(\alpha),\R(\beta)]$. If $u_0\le\R(\alpha)$ a.e in $\Om$, the maximum principle gives $\R(\R(\alpha)-u_0)\geq 0$ in $\Om$ and then $f_{opt}=\alpha$ satisfies (\ref{fopte3}) with $\lambda=0$. Analogously, if $u_0\ge\R(\beta)$ a.e. in $\Om$, then $f_{opt}=\beta$.

Assume $u_0\in \big[{\cal R}(\alpha),{\cal R}(\beta)]$ a.e. in $\Om$ and $u_0\not \equiv {\cal R}(\alpha)$, $u_0\not\equiv {\cal R}(\beta)$. If $f_{opt}=\alpha$ a.e. in $\Om$, the strong maximum principle gives $w>0$ in $\Om$ while 
$$\into f\,dx=\alpha|\Om|<m$$
implies $\lambda=0$. By (\ref{fopte3}) we conclude that $f_{opt}=\beta$, in contradiction with $f_{opt}=\alpha$. Similarly, if $f_{opt}=\beta$ a.e. in $\Om$, we get $w> 0$ a.e. in $\Om$ in contradiction with (\ref{fopte3}). Taking into account (\ref{fopte3}) we then deduce that $\big|\{w=\lambda\}\big|=0$ implies that $f_{opt}$ is a bang-bang control.

Another case in which $f_{opt}$ is of bang-bang type, again deduced from the necessary conditions of optimality \eqref{fopte3}, is when $-\Delta u_0\geq \beta$ a.e. in $\Om$ and $u_0\geq 0$ on $\partial\Om$.
\end{Exa}

\begin{Exa} Let us now consider an example where $\psi$ is strictly convex. By Remark \ref{coosbd} the optimal controls are not of bang-bang type. We take
$$\min\bigg\{\int_\Om j\big(x,\R(f),f\big)\,dx\ :\ \int_\Om f^2dx\le m\bigg\},\qquad m>0.$$
Now,
$$\psi(s)=s^2,\qquad\psi^*(s)={t^2\over 4},\qquad\partial\psi^*(t)={t\over2}.$$
Therefore, if $f$ is an optimal solution and $w$ is given by (\ref{defw}) we have the existence of $\lambda\geq 0$ such that
$$w=0\ \hbox{ a.e. in }\Om\qquad\hbox{or}\qquad f_{opt}={\sqrt{m}\,w^2\over \|w\|^2_{L^4(\Om)}}.$$
In the second case $f_{opt}$ is a continuous function by the summability assumptions on $j$ and their derivatives.
\end{Exa}

\begin{Exa}\label{compliance}
Consider the {\it compliance case}
\beq\label{compp}\min\bigg\{\int_\Om f\R(f)\,dx\ :\ \int_\Om f\,dx\ge m, \,\alpha\le f\le\beta \bigg\},\eeq
and assume $0\le\alpha<\beta$. To have a nontrivial problem we also assume $\alpha|\Om|<m<\beta|\Om|$. Using an integration by parts we have
$$\int_\Om f\R(f)\,dx=-2\E(f)$$
where $\E(f)$ is the energy
$$\E(f)=\min\bigg\{\int_\Om\Big(\frac12|\nabla u|^2-fu\Big)\,dx\ :\ u\in H^1_0(\Om)\bigg\},$$
and thus the optimization problem can be reformulated as
$$\max\bigg\{\E(f)\ :\ \int_\Om f\,dx\ge m,\ \alpha\le f\le\beta\bigg\}.$$
Similarly to example (\ref{exbang}) we have
$$\psi(s)=\begin{cases}
-s&\text{if }s\in[\alpha,\beta]\\
+\infty&\text{otherwise,}\end{cases}$$
$$\psi^*(t)=\begin{cases}
(t+1)\alpha &\text{if }t\leq -1\\
(t+1)\beta&\text{if }t\ge-1,\end{cases}\qquad \partial\psi^*(t)=\begin{cases}
\alpha&\text{if }t<-1\\
[\alpha,\beta]&\text{if }t=-1\\
\beta &\text{if }t>-1,
\end{cases}$$
and that $m$ in Theorem \ref{Thcop} must be chosen as $-m$.\par
Since $j(x,s,z)=sz$, we have that for a solution $f_{opt}$ of (\ref{compp}), the function $w$ defined by (\ref{defw}) is given by $w=2\R(f_{opt})$, where $f_{opt}\in [\alpha,\beta]$ implies $\R(f_{opt})$ strictly positive in $\Om$. Thus, Theorem \ref{Thcop} proves the existence of $\lambda>0$ such that
$$f_{opt}=\left\{\ba{ll}\dis \beta &\hbox{ if }\R(f_{opt})<\lambda\\ \ecart\dis \alpha &\hbox{ if }\R(f_{opt})>\lambda,\ea\right.\qquad
\into f_{opt}ds=m.$$
Moreover, as we saw in Remark \ref{coosbd} the set $\{\R(f_{opt})=\lambda\}$ has null measure. We are then in the bang-bang situation $f_{opt}=\alpha1_E+\beta1_{\Om\setminus E}$ for $E=\{\R(f_{opt})>\lambda\}$.
\end{Exa}

%%%%%%%%%%%%%%%%%%%%%%%%%%%%%%%%%%%%%%%%%%%%%%%%%%
\section{Regularity of optimal sources}\label{snreg}

We have seen in Section \ref{snec} that if the function $\psi$ in (\ref{optpb}) is not strictly convex, then the optimal solutions are of bang-bang type, where the interfaces are given by $\{w=s\}$, with $w$ defined by (\ref{defw}) and $s\in\RR$ (indeed, if this set has positive measure, then the optimal control could be continuous). The question we consider in the present section is to get some regularity results for bang-bang optimal solutions. Since they are discontinuous, we can ask whether they are $BV$ functions, that is, whether the set $\{w=s\}$ has a finite perimeter.

\subsection{$BV$ regularity}\label{ss:bv}

As a model problem, we can consider the compliance case of Example \ref{compliance}:
\beq\label{complper}
\min\bigg\{\int_\Om f\R(f)\,dx\ :\ \int_\Om f\,dx\ge m,\ f(x)\in[\alpha,\beta]\bigg\},
\eeq
with $0\le\alpha<\beta$ and $\alpha|\Om|<m<\beta|\Om|$. We have seen that the optimal solution $f_{opt}$ is of bang-bang type, that is
$$f_{opt}=\alpha1_E+\beta1_{\Om\setminus E}\qquad\hbox{with }E=\{\R(f_{opt})<s\},$$
for some positive constant $s$ that has to be chosen such that the integral constraint $\int_\Om f\,dx\ge m$ is saturated. The function $u=\R(f_{opt})$ thus solves te PDE
$$\begin{cases}
-\Delta u=\beta1_{\{u<s\}}+\alpha1_{\{u>s\}}&\text{in }\Om\\
u=0&\text{on }\partial\Om.
\end{cases}$$

\begin{Thm}\label{perimeter}
The optimal solution $f_{opt}$ of the minimization problem \eqref{complper} is in $BV(\Om)$, hence the optimal set $E$ above has a finite perimeter
\end{Thm}

\begin{proof}
It is enough to apply Theorem 3.5 of \cite{BCM}.
\end{proof}

\subsection{A weaker regularity}\label{ss:weaker}

Similarly to the above example, Theorem \ref{Thcop} and Remark \ref{coosbd} with $j(x,s,z)$ independent of $z$ prove that for bang-bang optimal controls, the interfaces are of the form $\{u=s\}$ with $u$ the solution of the PDE
$$\begin{cases}
-\Delta u=f&\hbox{in }\Om\\
u=0&\hbox{on }\partial\Om,
\end{cases}$$
where we set $f=\partial_sj\big(x,\R(f_{opt})\big)$. Some results about the regularity of the level sets of the solution of the above problem are simple to obtain. On the one hand, if $f\in L^q(\Om)$, with $q>d$ then $u$ is in $C^1(\overline \Om)$. Thus, the implicit function theorem proves that for every $s\in \RR$ the set
$$\{u=s\}\cap \{\nabla u\not =0\}$$
is a $C^1$ manifold. On the other hand, for $u$ just in $BV(\Om)$ the coarea formula (\cite{EvGa}, Chapter 5) gives
$$\into d|\nabla u|=\int_\RR\|\nabla 1_{\{u>s\}}\|_{{\cal M}(\Om)}\,ds.$$
Thus, except for $s$ in a subset of $\RR$ of null Lebesgue measure we have 
\beq\label{regbvls} 1_{\{u>s\}}\in BV(\Om).\eeq 
 
The question is now if, adding some assumptions on $f$, property (\ref{regbvls}) holds for every $s\in \RR$. Since the difficulties appear in the set $\{\nabla u=0\}$, let us assume that $f$ is positive in $\Om$ (by linearity, if $f$ is negative the argument is similar) in such way that this set has null Lebesgue measure.

The result below is slightly weaker than (\ref{regbvls}). We will only prove that for any $q>1$, we have
$$\log^{-q}\Big({1\over|\nabla u|}\vee e\Big) 1_{\{u>s\}}\in BV(\Om).$$
Observe that the factor $\log^{-q}\big({1/|\nabla u|}\vee e\big) $ vanishes on the ``bad set'' $\{\nabla u=0\}$ but it goes to zero very slowly with respect to $\nabla u$.

In the following, for a connected bounded open set $\Om\subset\RR^d$, $d\ge2$, we deal with $\R(f)$ solution of
\beq\label{EcLa}
\begin{cases}
-\Delta u=f&\text{in }\Om\\
u=0&\text{on }\partial\Om.
\end{cases}\eeq
We start with the following estimates for the solution of (\ref{EcLa})

\begin{Thm} \label{TeEst} 
Assume $\Om$ of class $C^{1,1}$; then for every $f\in BV(\Om)$, there exists $C>0$, which only depends on $\Om$ such that $u=\R(f)$ satisfies
\beq\label{est1} \into{1\over |\nabla u|}\Big|D^2u\Big(I-{\nabla u\otimes\nabla u\over |\nabla u|^2}\Big)\Big|^2dx\leq C\|f\|_{BV(\Om)},
\eeq 
\beq\label{est2}\int_{\{|\nabla u|<1/e\}} {|D^2u\nabla u|^2\over |\nabla u|^3\log^q\big({1\over|\nabla u|}\big)}dx
\leq {C\over q-1}\|f\|_{BV(\Om)},\quad\forall\, q>1. 
\eeq
Moreover, if $f$ satisfies 
\beq\label{positf} \exists \alpha>0\ \hbox{ such that }\ f\geq \alpha\ \hbox{ in }\Om.\eeq
then, for every $q>1$ and every $\ep>0$, we have
\beq\label{est3}\into {1\over |\nabla u|\log^q\big({1\over |\nabla u|}\vee e\big)}dx\leq C{q^2\over \alpha^2(q-1)}\|f\|_{BV(\Om)}+{H_{d-1}(\partial\Om)\over\alpha},\eeq
\beq\label{est4} {1\over \ep}\int_{\{s<u<s+\ep\}} {|\nabla u|\over \log^q\big({1\over |\nabla u|}\vee e\big)}\,dx\le C{q^2\over \alpha(q-1)}\|f\|_{BV(\Om)}+H_{d-1}(\partial\Om),\eeq
where $H_{d-1}$ denotes the $(d-1)$-Hausdorff measure in $\RR^N$.
\end{Thm}\noindent
\begin{proof}
It is enough to prove the result for $\Om$ of class $C^{2,\alpha}$, $\alpha>0$, $f\in C^{2,\alpha}(\overline \Om)$, and then $u\in C^{2,\alpha}(\overline \Om)$. The general case follows by an approximation argument, recalling that for every $f\in BV(\Om)$, there exists $f_n\in C^\infty(\overline\Om)$ such that
$$f_n\to f\ \hbox{ in }L^{d/(d-1)}(\Om),\quad \|\nabla f_n\|_{L^1(\Om)^d}\to \|\nabla f_n\|_{{\cal M}(\Om)^d},$$
and that the Calderon-Zygmund theorem implies that $u_n$ satisfies
$$ u_n\to u\ \hbox{ in }W^{2,{d/(d-1)}}(\Om).$$
In the following we define $\zeta:(0,\infty)\to \RR$ by
$$\zeta(s)=\begin{cases}
0&\text{if }s=0\\
\dis{1\over\log({1\over s}\vee e)}&\text{if }0<s.
\end{cases}$$
Let us prove (\ref{est1}), (\ref{est2}). We use that the derivatives of $u$ satisfy (see \cite{CaCoVa})
$$\begin{cases}
-\Delta\partial_i u=\partial_i f&\hbox{in }\Om,\quad1\le i\le d\\
\nabla u=-|\nabla u|\nu&\hbox{on }\partial\Om\\
-D^2u\,\nu\cdot\nu=f+h\cdot\nabla u&\hbox{on }\partial\Om,
\end{cases}$$
where $\nu=-\nabla u/|\nabla u|$ is the unitary outside normal to $\Om$, and $h$ is a function in $L^\infty(\partial\Om)^d$, depending only on $\Om$. For $\delta>0$ small enough, we take
$${\partial_i u\over |\nabla u|+\delta}\zeta\big(|\nabla u|+\delta\big)^{q-1}$$
as test function in the equation for $\partial_iu$. Summing with respect to the index $i$ and integrating by parts, we get
\beq\label{equ1}
\begin{split}
&\into{|D^2u|^2\over|\nabla u|+\delta}\zeta\big(|\nabla u|+\delta\big)^{q-1}dx-\into{|D^2u\nabla u|^2\over|\nabla u|(|\nabla u|+\delta)^2}\zeta\big(|\nabla u|+\delta\big)^{q-1}dx\\
&\qquad\qquad+(q-1)\int_{\{|\nabla u|+\delta<1/e\}}{|D^2u\nabla u|^2\over|\nabla u|(|\nabla u|+\delta)^2}\zeta\big(|\nabla u|+\delta\big)^qdx\\
&=\int_{\partial\Om}{|\nabla u|(f+h\cdot\nabla u)\over|\nabla u|+\delta}\zeta\big(|\nabla u|+\delta\big)^{q-1}dH_{d-1}(x)\\ &\qquad\qquad+\into{\nabla f\cdot\nabla u\over|\nabla u|+\delta}\zeta\big(|\nabla u|+\delta\big)^{q-1}dx.
\end{split}\eeq
The two first terms in the left-hand side can be written as
$$\into{\zeta\big(|\nabla u|+\delta\big)^{q-1}\over|\nabla u|+\delta}\bigg(\Big|D^2u\Big(I-{\nabla u\otimes\nabla u\over|\nabla u|^2}\Big)\Big|^2+{\delta|D^2u\nabla u|^2\over |\nabla u|^2(|\nabla u|+\delta)^2}\bigg)dx,$$
where the integrand is non-negative. Thus, we can use the monotone convergence theorem to pass to the limit as $\delta\to0$ in (\ref{equ1}) to get
\[\begin{split}
&\into{\zeta(|\nabla u|)^{q-1}\over|\nabla u|}\Big|D^2u\Big(I-{\nabla u\otimes\nabla u\over|\nabla u|^2}\Big)\Big|^2dx+(q-1)\int_{\{|\nabla u|<1/e\}} {|D^2u\nabla u|^2\over |\nabla u|^3}\zeta\big(|\nabla u|\big)^qdx\\
&=\int_{\partial\Om}(f+h\cdot\nabla u)\zeta(|\nabla u|)^{q-1}dH_{d-1}(x)+\into{\nabla f\cdot\nabla u\over|\nabla u|}\zeta(|\nabla u|)^{q-1}dx.
\end{split}\]
Using $\zeta\le1$ and that $u\in W^{2,{d/(d-1)}}(\Om)$ implies $\nabla u\in L^1(\partial\Om)^d$ we deduce (\ref{est2}). Inequality (\ref{est1}) follows letting $q\to1$ in the above equality.\par\medskip

Let us now prove (\ref{est3}), (\ref{est4}). We multiply (\ref{EcLa}) by 
$${\zeta^q(|\nabla u|+\delta)\over |\nabla u|+\delta},$$
with $\delta>0$ and then we integrate in $\{u<t\}$, for $t>0$, such that  $\{u=t\}$ is a $C^1$ manifold (this holds for ever $t$
outside a subset of $(0,\infty)$ with null measure). We get
\[\begin{split}
&\int_{\{u<t\}}{D^2u\nabla u\cdot \nabla u\over |\nabla u|\big(|\nabla u|+\delta\big)^2}\zeta^q(|\nabla u|+\delta) \Big(-1+q\zeta(|\nabla u|+\delta)1_{\{|\nabla u|+\delta<{1\over e}\}}\Big)\Big)\,dx\\
&\qquad\qquad+\int_{\partial\Om}{|\nabla u|\zeta^q(|\nabla u|+\delta)\over |\nabla u|+\delta}\,dH_{d-1}(x)\\
&=\int_{\{u=t\}}{|\nabla u|\zeta^q(|\nabla u|+\delta)\over |\nabla u|+\delta}\,dH_{d-1}(x)+\int_{\{u<t\}}f{\zeta^q(|\nabla u|+\delta)\over |\nabla u|+\delta}\,dx.
\end{split}\]
Using (\ref{positf}) in the last term and Young's inequality in the first one, this gives
\[\begin{split}
&\int_{\{u=t\}}{|\nabla u|\zeta^q(|\nabla u|+\delta)\over|\nabla u|+\delta}\,dH_{d-1}(x)+{\alpha\over 2}\int_{\{u<t\}}{\zeta^q(|\nabla u|+\delta)\over|\nabla u|+\delta}\,dx\\
&\le{1\over 2\alpha}\int_{\{u<t\}}{|D^2u\nabla u|^2\over (|\nabla u|+\delta)^3}\zeta^q(|\nabla u|+\delta)\Big(-1+q\zeta(|\nabla u|+\delta)1_{\{|\nabla u|+\delta<{1\over e}\}}\Big)^2dx\\
&\qquad\qquad+\int_{\partial\Om}{|\nabla u|\zeta^q(|\nabla u|+\delta)\over |\nabla u|+\delta}\,dH_{d-1}(x).
\end{split}\]
Thanks to (\ref{est2}) we  deduce that this inequality holds for every $t>0.$ Moreover, it allows us  to pass to the limit when  $\delta\to 0$ using  the Lebesgue's dominated convergence theorem in  right-hand side  and the monotone convergence theorem in the left-hand side. Thus, we get
\[\begin{split}
&\int_{\{u=t\}}\zeta^q(|\nabla u|)\,dH_{d-1}(x)+{\alpha\over 2}\int_{\{u<t\}}{\zeta^q(|\nabla u|)\over|\nabla u|}\,dx\\
&\le{1\over 2\alpha}\int_{\{u<t\}}{|D^2u\nabla u|^2\over|\nabla u|^3}\zeta^q(|\nabla u|)\Big(-1+q\zeta(|\nabla u|)1_{\{|\nabla u|<{1\over e}\}}\Big)^2dx\\
&\qquad\qquad+\int_{\partial\Om}\zeta^q(|\nabla u|)\,dH_{d-1}(x),
\end{split}\]
and then, by (\ref{est2}) and $0\le\zeta\le1$, that there exists $C>0$ satisfying
\beq\label{equ2ep}
\begin{split}
&\int_{\{u=t\}}\zeta^q(|\nabla u|)\,dH_{d-1}(x)+{\alpha\over2}\int_{\{u<t\}}{\zeta^q(|\nabla u|)\over|\nabla u|}\,dx\\
&\qquad\qquad\le{Cq^2\over \alpha(q-1)}\|f\|_{BV(\Om)}+H_{d-1}(\partial\Om).
\end{split}\eeq
Estimate (\ref{est3}) just follows from this inequality taking $t$ tending to infinity. \par
To get estimate (\ref{est3}) we recall the coarea formula for Lipschitz functions (\cite{EvGa}, Chapter 5) which establishes \beq\label{equ3ep}\int_{\RR}g|\nabla u|\,dx=\int_\RR\int_{\{u=t\}}g\,dH_{d-1}(x)dt,\quad \forall\, g\in L^1(\Om).\eeq\par
Using (\ref{equ3ep}) with $g=\zeta^q(|\nabla u|) 1_{\{s<u<s+\ep\}}$ we get
(\ref{est4}) from (\ref{equ2ep}).
\end{proof}

\begin{Cor} \label{Correil} For $\Om\in C^{1,1}$ and $f\in BV(\Om)$, satisfying (\ref{positf}), the function $u=\R(f)$ is such that
$$z:={1\over \log^q\big({1\over |\nabla u|}\vee e\big)}$$
belongs to $W^{1,1}(\Om)$, for every $q>0$, and there exits $C>0$ depending only on $\Om$ such that
\beq\label{estgw}\|\nabla z\|_{L^1(\Om)^N}\leq C\Big({q+1\over \alpha q}\|f\|_{BV(\Om)}+H_{d-1}(\partial\Om)\Big).\eeq
\end{Cor}

\begin{proof}
Taking into account
\[\begin{split}
|\nabla z|&={|D^2 u\nabla u|\over |\nabla u|^2\log^{q+1}\big({1\over |\nabla u|}\vee e\big)}1_{\{|\nabla u|<1/e\}}\\
&={|D^2 u\nabla u|\over |\nabla u|^{3\over 2}\log^{q+1\over 2}\big({1\over |\nabla u|}\vee e\big)}\ {1\over |\nabla u|^{1\over 2}\log^{q+1\over 2}\big({1\over |\nabla u|}\vee e\big)}1_{\{|\nabla u|<1/e\}},
\end{split}\]
Using the Cauchy-Schwarz inequality the result follows from (\ref{est2}) and (\ref{est3}) with $q$ replaced by $q-1$.
\end{proof}

Our main result about the regularity of the function $1_{\{u>s\}}$ is given by

\begin{Thm}\label{Thapls} Assume $\Om$ of class $C^{1,1}$ and let $f\in BV(\Om)$ satisfying (\ref{positf}). Then the function $u=\R(f)$ satisfies for every $s>0$ and every $q>1$ 
\beq\label{estdfc}
{1\over\log^q\big({1\over|\nabla u|}\vee e\big)}1_{\{u>s\}}\in BV(\Om).\eeq
Moreover
\beq\label{estgu>s}{1\over \log^q\big({1\over|\nabla u|}\vee e\big)}\nabla 1_{\{u>s\}}\in {\cal M}(\Om),\eeq
and there exits $C>0$ only depending on $\Om$ such that
\beq\label{est1dlg}
\Big\|{1\over \log^q\big({1\over|\nabla u|}\vee e\big)}\nabla 1_{\{u>s\}}\Big\|_{{\cal M}(\Om)^d}\leq C{q^2\over\alpha(q-1)}\|f\|_{BV(\Om)}+H_{d-1}(\partial\Om).\eeq
\end{Thm}

\begin{proof}
We fix $s>0$ and $q>1$, then, for $\ep>0$, we define
$$v_\ep :={T_\ep(u-s)^+\over \ep\log^q\big({1\over |\nabla u|}\vee e \big)}.$$
By the Lebesgue dominated convergence theorem, we have
$$v_\ep\rightarrow {1\over \log^q\big({1\over|\nabla u|}\vee e\big)}1_{\{u>s\}}\ \hbox{ in }L^p(\Om),\ \forall\, p\in [1,\infty).$$
Moreover,
\beq\label{convgve}\nabla v_\ep={\nabla u\over \ep}1_{\{s<|u|<s+\ep\}}{1\over \log^q\big({1\over|\nabla u|}\vee e\big)}+{T_\ep(u-s)^+\over\ep}\nabla \Big({1\over \log^q\big({1\over|\nabla u|}\vee e\big)}\Big),\eeq
where the right-hand side is bounded in $L^1(\Om)^N$ by (\ref{est4}) and (\ref{estgw}). This proves (\ref{estdfc}).

Assertion (\ref{estgu>s}) also comes from (\ref{convgve}), which gives
$$\ba{l}\dis{\nabla u\over \ep}1_{\{s<|u|<s+\ep\}}{1\over \log^q\big({1\over|\nabla u|}\vee e\big)}=\nabla v_\ep-{T_\ep(u-s)^+\over\ep}\nabla \Big({1\over \log^q\big({1\over|\nabla u|}\vee e\big)}\Big)\\ \ecart\dis \stackrel{*}\rightharpoonup
\nabla \Big({1\over \log^q\big({1\over|\nabla u|}\vee e\big)} 1_{\{u>s\}}\Big)-1_{\{u>s\}}\nabla \Big({1\over \log^q\big({1\over|\nabla u|}\vee e\big)}\big)\\ \ecart\dis={1\over \log^q\big({1\over|\nabla u|}\vee e\big)}\nabla 1_{\{u>s\}}\ \hbox{ in }{\cal M}(\Om)^N,\ea$$
taking into account that the left-hand side is bounded in $L^1(\Om)$ by (\ref{est4}). Inequality (\ref{est1dlg}) is also a consequence of the estimate of the left-hand side by (\ref{est4}).
\end{proof}

\begin{Rem} \label{obmn0} As we said at the beginning of subsection \ref{ss:weaker}, assumption (\ref{positf}) implies that the set $\{\nabla \R(f)=0\}$ has null measure. A further result is given by (\ref{est3}) which proves that $\big(|\nabla u|\log^q(1/|\nabla u|)\big)^{-1}$ is integrable for $q>1$ and $\nabla u$ close to zero. Observe that this result does not extend to $q=1$. For example, taking $f=1$ and $\Om$ the annulus $B(0,2)\setminus \overline B(0,1)$ we have
$$\nabla u=\begin{cases}
\dis{1\over 2}\Big(-|x|+{3\over 2\log 2\,|x|}\Big){x\over |x|}&\text{if }d=2\\ \noalign{\smallskip}
\dis{1\over d}\Big(-|x|+{3(d-2)2^{d-2}\over 2(2^{d-2}-1)|x|^{d-1}}\Big){x\over |x|}&\text{if }d>2.
\end{cases}$$
Thus, using that $\nabla u$ vanishes on $\{|x|=r\}$ for some $r\in (1,2)$, we easily get
$$\int_\Om{1\over |\nabla u|\log^q({1\over |\nabla u|}\vee e)}dx<\infty \iff q>1.$$
\end{Rem}

Estimate (\ref{est3}) allows us to prove that for every $f\in W^{1,p}(\Om)$, $p>d$, which satisfies (\ref{positf}), the Hausdorff dimension of the set $\{\nabla \R(f)=0\}$ is at most $d-1$. However, we are not able to prove $H_{d-1}(\{\nabla u=0\}) <\infty$ as in the example in Remark \ref{obmn0}. In order to give a more accurate result, we introduce the following refinement of the usual $H_{d-1}$-measure.

\begin{Def} For $q\geq 0$ and $A\subset\RR^d$, we define
$$H^\delta_{d-1,q}(A)=\inf\left\{\sum_{i=1}^n{r_i^{d-1}\over\log^q\big({1\over r_i}\big)}:\quad A\subset \bigcup_{i=1}^n B(x_i,r_i),\ r_i<\delta\right\},\quad 0<\delta<1,$$
and
$$H_{d-1,q}(A)=\lim_{\delta\to 0}H^\delta_{d-1,q}(A)=\sup_{\delta>0}H^\delta_{d-1,q}(A).$$
\end{Def}

\begin{Rem} Clearly, $H_{d-1,q}$ is an outer measure. It agrees with the usual $(d-1)$-Hausdorff measure for $q=0$ and satisfies
$$H_{d-1,q}(A)=0\ \hbox{ for some }q\ge0\Longrightarrow H_s(A)=0,\ \forall \,s>d-1,$$
with $H_s$ the $s$-Hausdorff measure. Thus, every set $A$ with $H_{d-1,q}(A)=0$ for some $q\geq 0$ has Hausdorff dimension at most $d-1$.
\end{Rem}

\begin{Thm} \label{teoec} Assume $\Om\in C^{1,1}$ and $f\in W^{1,p}(\Om)$, with $p>d$ such that (\ref{positf}) is satisfied. Then, for every $q>1$ the solution $u$ of (\ref{EcLa}) satisfies
\beq\label{estcap} H_{d-1,q}\big(\{\nabla u=0\})=0.\eeq
\end{Thm}

\begin{proof} We take $A:=\{\nabla u=0\}$. 
By (\ref{est3}) and $|A|=0$, for every $\ep>0$ there exists an open set $U\subset\Om$ with $A\subset U$ and
$$\int_U{1\over |\nabla u|\log^q\big({1\over |\nabla u|}\vee e\big)}dx<\ep.$$\par
Let $\delta\in (0,1)$ be. Using that $A$ is compact, we can find $x_i\in A$ and $0<r_i<\delta$, $1\leq i\leq m$, such that
\beq\label{esc1}A\subset\bigcup_{i=1}^m B(x_i,r_i),\qquad\overline{B}(x_i,r_i)\subset U,\ 1\leq i\leq m.\eeq
By the Vitali's covering theorem we can now extract $n$ balls $B(x_{i_j},r_{i_j})$, $1\leq j\le n$, which are disjoint and satisfy
\beq\label{esc2}\bigcup_{i=1}^m \overline{B}(x_i,r_i)\subset\bigcup_{j=1}^n\overline{B}(x_{i_j},5r_{i_j}).\eeq
On the other hand, since $f\in W^{1,p}(\Om)$, $p>d$ implies that $\nabla u$ is Lipschitz and $\nabla u(x_{i_j})=0$, there exists $L>0$ such that
$$|\nabla u(x)|\le L|x-x_{i_j}|,\qquad\forall x\in\Om,\ 1\le j\le n.$$
Then, for a certain constant $c>0$, we have
$$\int_{B(x_{i_j},r_{i_j})}{1\over |\nabla u|\log^q\big({1\over|\nabla u|}\vee e\big)}\,dx\ge c\int_0^{r_{i_j}}{r^{d-2}\over\log^q(1/r)}\,dr,$$
where an integration by parts gives
$$\int_0^{r_{i_j}}{r^{d-2}\over\log^q(1/r)}\,dr
={r_{i_j}^{d-1}\over (d-1)\log^q(1/r_{i_j})}+
{1\over d-1}\int_0^{r_{i_j}}{r^{d-2}\over \log^{q+1}(1/r)}\,dr,$$
and then, assuming $\delta$ small enough and recalling that $r_{i_j}<\delta$, we get
$$\int_0^{r_{i_j}}{r^{d-2}\over \log^q(1/r)}\,dr\ge{r_{i_j}^{d-1}\over 2(d-1)\log^q(1/r_{i_j})}.$$
Using that 
$$c\sum_{j=1}^n\int_0^{r_{i_j}}{r^{d-2}\over\log^q(1/r)}dr\le\int_U{1\over |\nabla u|\log^q\big({1\over |\nabla u|}\vee e\big)}dx<\ep,$$
we deduce
$$\sum_{j=1}^n{r_{i_j}^{d-1}\over\log^q(1/r_{i_j})}\le{2(d-1)\over c}\,\ep.$$
By (\ref{esc1}) and (\ref{esc2}) we then have
$$H_{d-1,q}^\delta(A)\leq \sum_{j=1}^n{(5r_{i_j})^{d-1}\over\log^q(1/(5r_{i_j}))}\le{5^{d-1}2(d-1)\over c}\,\ep,$$
which by the arbitrariness of $\ep$ proves (\ref{estcap}).
\end{proof}

\subsection{The case $\Om$ convex}\label{ss:convex}

When the domain $\Om$ is convex, in some cases we can obtain a better regularity for the optimal right-hand side $f_{opt}$. Let us return to the compliance case
$$\min\bigg\{\int_\Om f\,\R(f)\,dx\ :\ \int_\Om f\,dx\ge m,\ 0\le f\le1\bigg\}$$
with $0<m<|\Om|$, and assume $\Om$ convex. We have seen in Example \ref{compliance} that the optimal right-hand side $f_{opt}$ is of bang-bang type: $f_{opt}=1_E$ with $E=\{w<s\}$ for a suitable $s$ such that $|E|=m$, where $w$ is the solution of the PDE
\beq\label{eqEconvex}
\begin{cases}
-\Delta w=1_{\{w<s\}}&\text{in }\Om\\
w=0&\text{on }\partial\Om.
\end{cases}\eeq

\begin{Lem}\label{Econvex}
The set $E=\{w<s\}$ above is convex.
\end{Lem}

\begin{proof}
It is enough to apply Theorem 1.2 of \cite{BMS}. In fact this theorem applies to solutions $v$ of
$$\begin{cases}
-\Delta v=\phi(v)&\text{in }\Om\\
v=0&\text{on }\partial\Om
\end{cases}$$
with $\phi$ H\"older continuous such that
\begin{itemize}
\item[(i)]$\sqrt\Phi$ is concave,
\item[(ii)]$\Phi/\phi$ is convex on $]0,M[$,
\end{itemize}
where $\Phi$ is the primitive of $\phi$ with $\Phi(0)=0$, and
$$M=\inf\big\{t>0\ :\ \phi(t)=0\big\}.$$
By approximating our function $\phi=1_{[0,s]}$ by
$$\phi_n(t)=\begin{cases}
1-(t/s)^n&\text{if }t\le s\\
0&\text{if }t>s
\end{cases}$$
we see that $\phi_n$ satisfies conditions (i) and (ii), hence the level sets of the functions $v_n$ solutions of the PDE
$$\begin{cases}
-\Delta v=\phi_n(v)&\text{in }\Om\\
v=0&\text{on }\partial\Om
\end{cases}$$
are convex. Passing to the limit as $n\to\infty$, we have that the level sets of the solution $v$ are convex too.
\end{proof}

\begin{Pro}
The set $E$ is of class $C^1$.
\end{Pro}

\begin{proof}
By Lemma \ref{Econvex} the set $E$ is convex; assume by contradiction that it has a corner. The solution $w$ of \eqref{eqEconvex} satisfies the PDE
$$\begin{cases}
-\Delta w=1\text{ in }\Om\setminus E\\
w=0\text{ on }\partial\Om,\quad w=s\text{ on }\partial E;
\end{cases}$$
in addition, by \eqref{eqEconvex} we have that $w$ is $W^{2,p}$ regular near the corner for every $p$, which is impossible by the well-known theory of elliptic PDEs in domains with re-entrant corners.
\end{proof}

%%%%%%%%%%%%%%%%%%%%%%%%%%%%%%%%%%%%%%%%%%%%%%%%%%
\section{Numerical simulations}\label{snum}

In this section we show some numerical examples, in the two-dimensional case, for problem (\ref{optpb}). We consider three cases:
\begin{itemize}
\item[-] Problem \eqref{maxlp} relative to the
maximization of the $L^p$ norm of $\R(f)$ when $f$ is non-negative and has a bounded mass;
\item[-]The minimization problem (\ref{pbme3}) in Example \ref{exbang} in the case of a linear cost $j(x,s)=g(x)\,s$ for some suitable function $g$;
\item[-]The minimization problem (\ref{pbme3}) in Example \ref{exbang} in the quadratic case $j(x,s)=|s-u_0(x)|^2$ for some suitable function $u_0$.
\end{itemize} 

We apply a gradient descent method derived from an appropriate use of the optimality conditions given by Theorem \ref{Thcop}. 
We refer to \cite{All,BeSig,Cas} for other algorithms related to similar problems. The algorithm is as follows.

\begin{itemize}
\item Initialization: choose an admisible function $f_0\in L^1(\Om)$.
\item For $n\ge0$, iterate until stop condition as follows.
\begin{itemize}
\item Compute  $w_n$ as in (\ref{defw}) for $f_{opt}=f_n$.
\item Compute $\hat f_n$ descent direction associated as:
\begin{itemize}
\item Example 4.5 
$$\hat f_n(x)= m \delta_{x_n}
$$
with $x_n$ the point where the minimum of $w_n$ is attained.
\item Example 4.6
$$\hat f_n(x)=\left\{\ba{ll}
\beta &\hbox{ if }  w_n(x)<-\lambda_n,\\ \alpha &\hbox{ in other case,}
\ea\right.$$
where $\lambda_n$ is the Lagrange multiplier associated to the volume constraint.
\end{itemize}
\item For $\ep_n\in[0,1)$ small enough, update the function $f_n$:
$$f_{n+1}=f_n+\epsilon_n (\hat f_n-f_n).$$

\end{itemize}
\item Stop if $\frac{|I_n-I_{n-1}|}{|I_0|}<tol$, for $tol>0$ small, with
$$I_n=\into\big(j(x,\R(f_n))+\psi(f_n)\big)\,dx,\quad n\geq 0.$$
\end{itemize}

The computation has been carried out using the free software FreeFem++ v4.5 (\cite{Hec}, available in {\tt http://www.freefem.org}). The picture of figures are made in Paraview 5.10.1 (available at {\tt https://www.kitware.com/open-source/\# paraview}), which is free too, except  Figure \ref{Fig:F1} which is made with MATLAB. We use P1-Lagrange finite element approximations for the control function  $f$, the state  $\R(f)$ and costate $w$. For all simulations of the Example 4.6 where the parameters $\alpha$ and $\beta$ appear, we consider the normalized values $\alpha=0$ and $\beta=1$.

\begin{Exa}
We consider the maximization problem
$$\max\bigg\{\int_\Om|\R(f)|^p\,dx\ :\ f\geq 0,\quad \into f\,dx\le m\bigg\}$$
in dimension two, with $p=4$ and volume constraint $m=10$. The domain $\Om$ is a ball with a non-centered hole and a sharp mesh with 87806 triangles, see Figure \ref{Fig:mesh}. According to  the analysis of optimality condition made in Example \ref{ex45}, the optimal right-hand side $f_{opt}=m\delta_{x_0}$ is a Dirac mass where the point $x_0$ is explicitly computed by  $(-0.429729, 0.212863)$, see Figure \ref{Fig:F1}. In Figure \ref{Fig:cost} we can observe the decreasing cost evolution for the minimization algorithm.

\begin{figure}[h!]\centering
\includegraphics[width=0.8\textwidth]{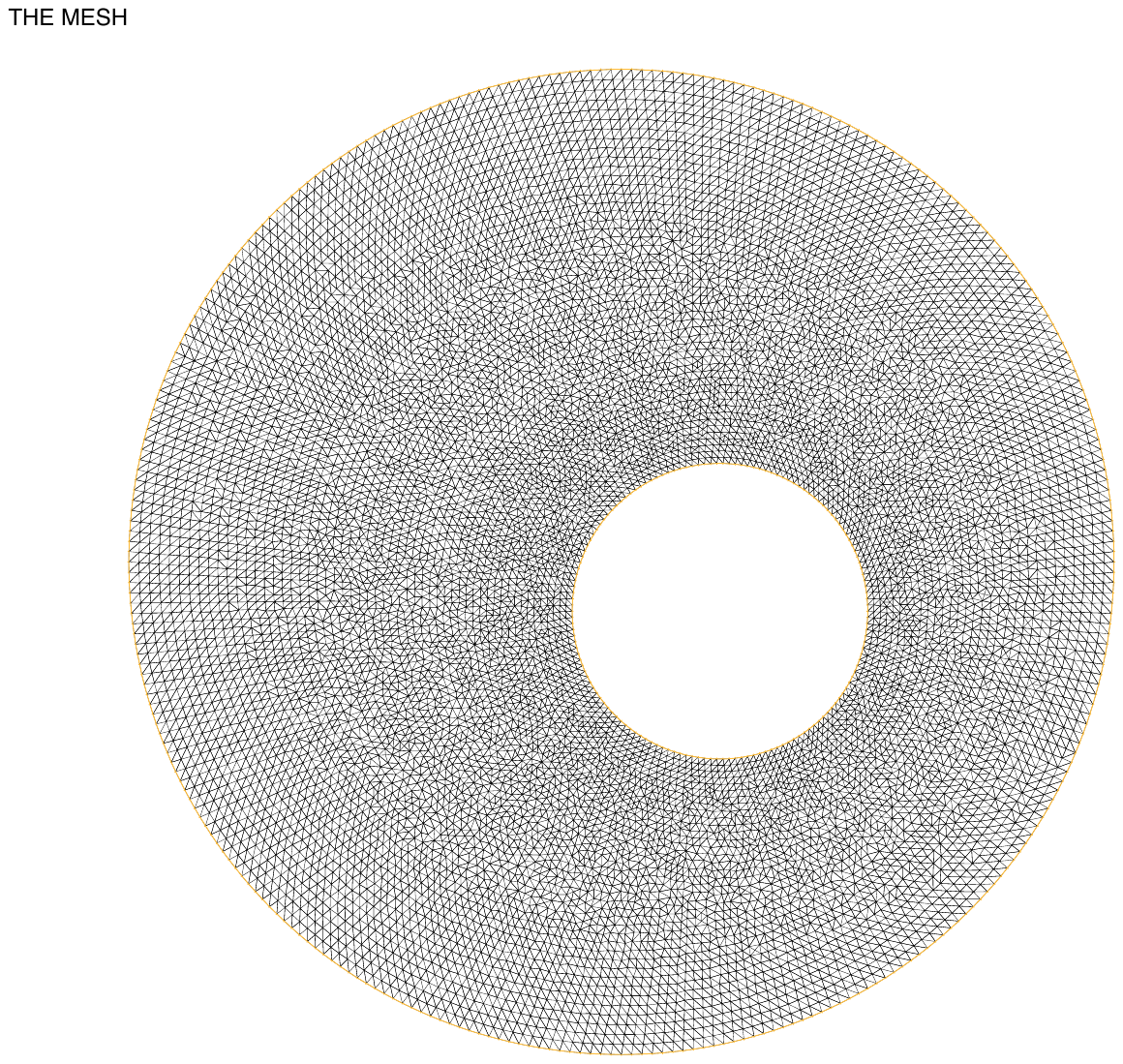}
\vskip-0,50cm
\caption{First numerical simulation: the mesh.}
\label{Fig:mesh}
\end{figure}

\begin{figure}[h!]\centering
\includegraphics[width=0.5\textwidth]{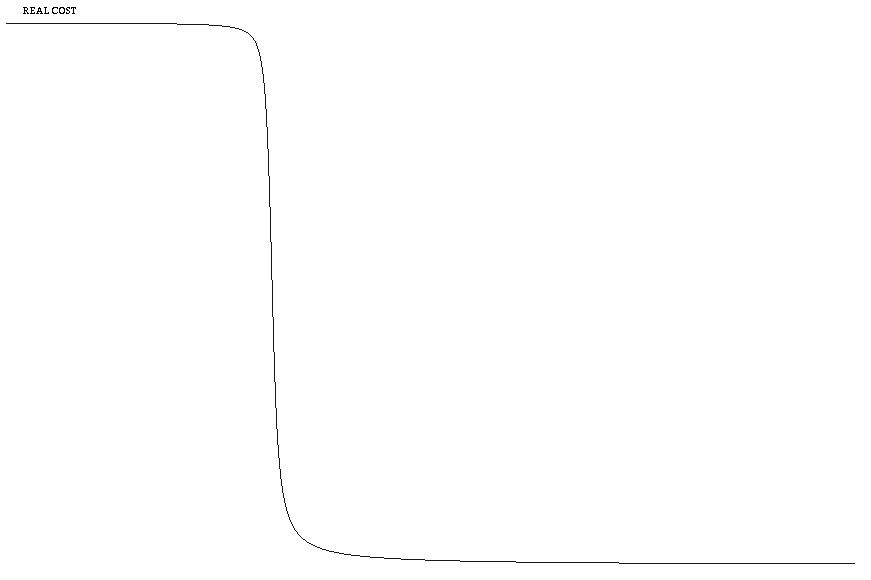}
\caption{First numerical simulation: cost evolution.}
\label{Fig:cost}
\end{figure}

\begin{figure}[h!]\centering
\includegraphics[width=.7\textwidth]{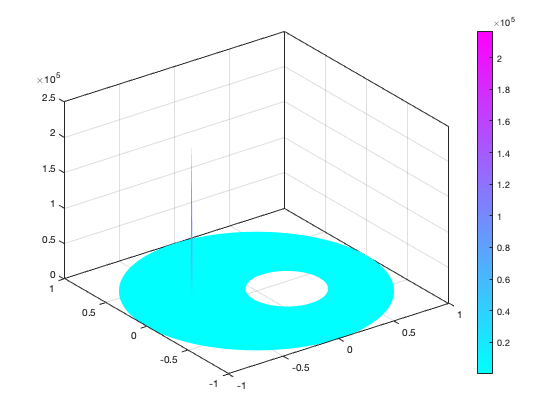}
\vskip-.5cm
\caption{First numerical simulation: the optimal right-hand side $f_{opt}=m\delta_{x_0}$.}
\label{Fig:F1}
\end{figure}
\end{Exa}
%%%%%%%%%%%%%%%%%%%%%%%%%%%%%%%%%%%%%%%%%%%%%%%%%%

\begin{Exa}
We solve numerically the problem (\ref{pbme3}) for $\Om$ the unit ball of $\RR^2$ and the linear cost given by $j(x,s,z)=g(x)\,s$ with $g(x,y)=x^2-y^2$. We take $m=1.25$ corresponding to use, approximately, a maximum of 40\% of $\beta$. We can observe the computed optimal right-hand side $f_{opt}$ in Figure \ref{Fig:F2}.

\begin{figure}[h!]\centering
\includegraphics[width=.7\textwidth]{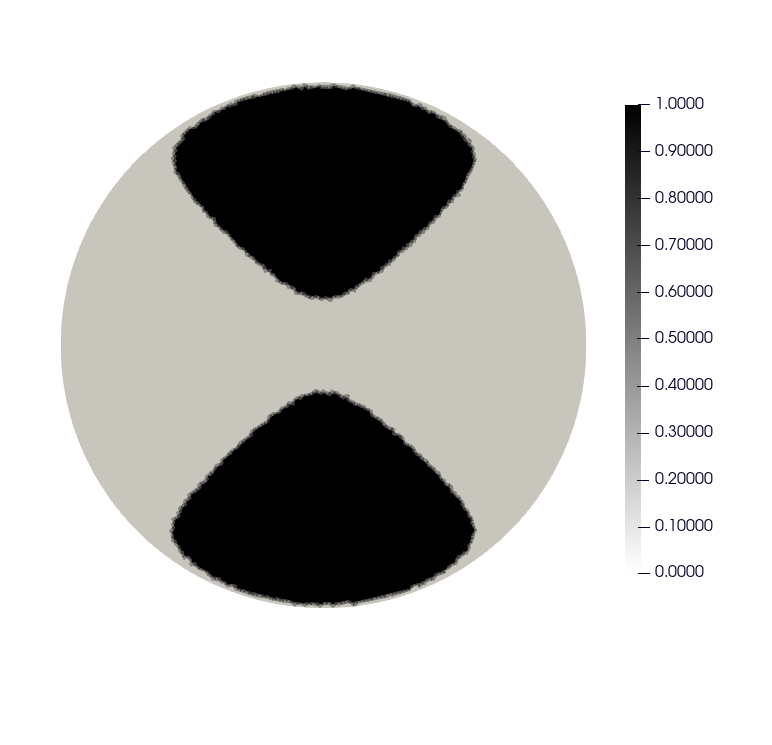}
\vskip-1.75cm
\caption{Second numerical simulation: the optimal right-hand side $f_{opt}$.}
\label{Fig:F2}
\end{figure}
\end{Exa}
%%%%%%%%%%%%%%%%%%%%%%%%%%%%%%%%%%%%%%%%%%%%%%%%%%

\begin{Exa}
In this last example we solve numerically also, the problem (\ref{pbme3}) for $\Om$  the unit ball of $\RR^2$ and $m=1.25$ as in the previous case, but we consider $j(x,s,z)=|s-u_0|^2\ $ taking a constant function $u_0=0.1$. As we can expect $f_{opt}$ is a bang-bang control, see Figure \ref{Fig:F3}.

\begin{figure}[h!]\centering
\includegraphics[width=.8\textwidth]{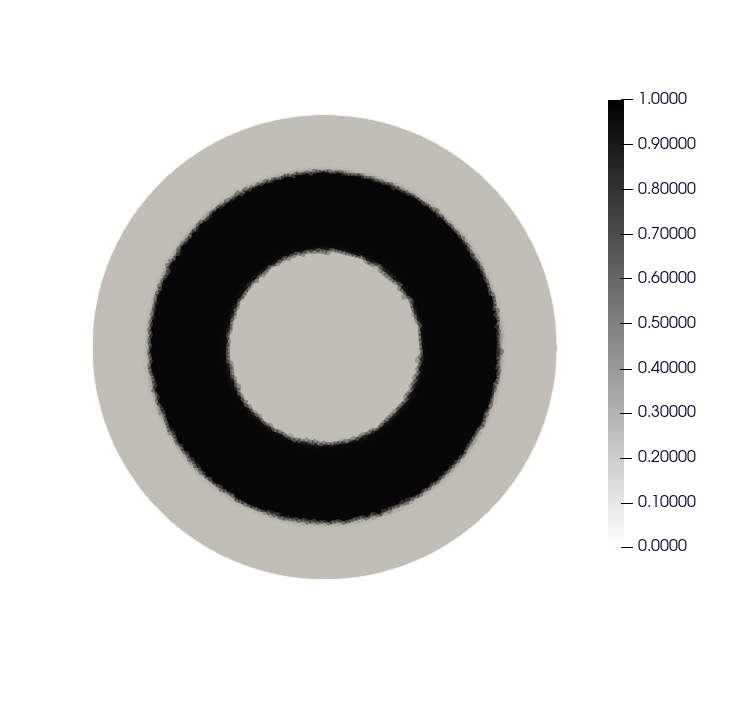}
\vskip-1.75cm
\caption{Third numerical simulation: $f_{opt}$ is bang-bang.}
\label{Fig:F3}
\end{figure}
\end{Exa}

%%%%%%%%%%%%%%%%%%%%%%%%%%%%%%%%%%%%%%%%%%%%%%%%%%
\bigskip
\noindent{\bf Acknowledgments. }The work of GB is part of the project 2017TEXA3H {\it``Gradient flows, Optimal Transport and Metric Measure Structures''} funded by the Italian Ministry of Research and University. GB is member of the Gruppo Nazionale per l'Analisi Matematica, la Probabilit\`a e le loro Applicazioni (GNAMPA) of the Istituto Nazionale di Alta Matematica (INdAM).\par 
The work of JCD and FM is a part of the FEDER project PID2023-149186NB-I00 of the {\it Ministerio de Ciencia, Innovaci\'on y Universidades} of the government of Spain.
\bigskip
%%%%%%%%%%%%%%%%%%%%%%%%%%%%%%%%%%%%%%%%%%%%%%%%%%

\end{document}